\documentclass[12pt,reqno]{amsart}
\usepackage[utf8]{inputenc}
\usepackage[T1]{fontenc}
\usepackage[usenames, dvipsnames]{color}
\usepackage{ulem}

\usepackage{dsfont, amsfonts, amsmath, amssymb,amscd, stmaryrd, latexsym, amsthm, dsfont}
\usepackage[frenchb,english]{babel}
\usepackage{enumerate}
\usepackage{longtable}
\usepackage{geometry}
\usepackage{float}
\usepackage{tikz}
\usetikzlibrary{shapes,arrows}
\usepackage{float}
\usepackage{tikz}
\usepackage{xypic}
\usetikzlibrary{shapes,arrows}

\newtheorem{theorem}{Theorem}[section]
\newtheorem{lemma}[theorem]{Lemma}

\theoremstyle{remark}

\usepackage{hyperref}
\hypersetup{
	colorlinks=true,
	urlcolor=blue,
	citecolor=blue}
\def\NN{\mathds{N}}
\def\RR{\mathbb{R}}
\def\QQ{\mathbb{Q}}
\def\CC{\mathbb{C}}
\def\ZZ{\mathbb{Z}}
\def\kk{\mathds{k}}
\begin{document}
	
\def\NN{\mathbb{N}}
\def\RR{\mathds{R}}
\def\HH{I\!\! H}
\def\QQ{\mathbb{Q}}
\def\CC{\mathds{C}}
\def\ZZ{\mathbb{Z}}
\def\DD{\mathds{D}}
\def\OO{\mathcal{O}}
\def\kk{\mathds{k}}
\def\KK{\mathbb{K}}
\def\ho{\mathcal{H}_0^{\frac{h(d)}{2}}}
\def\LL{\mathbb{L}}
\def\L{\mathds{k}_2^{(2)}}
\def\M{\mathds{k}_2^{(1)}}
\def\k{\mathds{k}^{(*)}}
\def\l{\mathds{L}}

\selectlanguage{english}
\title[The $2$-Iwasawa module...]{The $2$-Iwasawa module over certain octic
	elementary fields}

\author[M. M. Chems-Eddin]{Mohamed Mahmoud Chems-Eddin}
\address{Mohamed Mahmoud CHEMS-EDDIN: Mohammed First University, Mathematics Department, Sciences Faculty, Oujda, Morocco }
\email{2m.chemseddin@gmail.com}

\subjclass[2010]{ 11R29; 11R23;   11R18; 11R20.}
 
  	\keywords{Cyclotomic $Z_p$-extension, $2$-Iwasawa module, $2$-rank, $2$-class group.}
 
 \maketitle
 \begin{abstract} 
 		The aim of this paper is to determine the structure of  $2$-Iwasawa module of some imaginary triquadratic fields.
 \end{abstract}

 \section{Introduction}
 \label{Sec:1}
 
 Let $\ell$ denote   a prime number. Let $\ZZ_\ell$  be the additive group of $\ell$-adic integers,  $\Lambda=\mathbb Z_\ell \llbracket  T \rrbracket$ and 	$M$  a finitely generated   $\Lambda$-module.   It is known, by the structure theorem,  that  there exist  $r$, $s$,  $t$, $n_i$, $m_i  \in\mathbb Z$,  and  $f_i$   distinguished and irreducible polynomials  of  $\mathbb Z_\ell[T]$ such that 
 \begin{eqnarray*}
 	M\sim \Lambda^r \oplus \left( \bigoplus_{i=1}^{s} \Lambda/ (\ell^{n_i})\right)\oplus \left(  \bigoplus_{j=1}^{t}\Lambda/(f_j(T)^{m_j}) \right).
 \end{eqnarray*}

 Denote   by $k_\infty/k$ the cyclotomic  $\mathbb Z_\ell$-extension   of a number field $k$.
 For any integer $n\geq 1$, it is known that  $k_\infty$ contains a unique cyclic subfield $k_n$  of degree $\ell^n$ over $k$. The field $k_n$ is called the $n$-th	layer of 
 the cyclotomic $\mathbb Z_\ell$-extension of $k$.  Let $A_n(k)$  denote the $\ell$-class group of $k_n$ and us consider the Iwasawa module  $A_\infty(k)=\varprojlim A_n(k)$ which is a finitely generated $\Lambda$-module.
 
 The determination of the structure of $A_\infty(k)$, for a given number field $k$,    is a very difficult and interesting  problem in Iwasawa theory.	
 In this paper we shall deal with this problem (in the case $\ell =2$) for  some  infinite families of number fields of the form $F=\QQ(\sqrt{p},\sqrt{q}, i)$, where $p$ and $q$ are two primes  satisfying one of the following conditions:

 \begin{eqnarray}\label{CONDD1}
	q\equiv7\pmod 8, \;p\equiv 5 \text{ or } 3\pmod 8 \text{ and }   \left(\dfrac{p}{q}\right)=1.
\end{eqnarray}
 \begin{eqnarray}\label{COND2D}
 q\equiv7\pmod 8, \;p\equiv 5 \text{ or } 3\pmod 8 \text{ and }   \left(\dfrac{p}{q}\right)=-1.
 \end{eqnarray}

 Let us adopt the following notations:    $h_2(d)$   the $2$-class number  of $\QQ(\sqrt{d})$, $h(k)$ (resp. $h_2(k)$) the class number  (resp.   $2$-class number) of a number field $k$, $k^+$ the maximal real subfield of $k$, $\mathrm{Cl}(k)$ the class group of $k$,    $q(k)$ the unit index of a multiquadratic number field $k$ and $N_{k'/k}$ the norm map of an extension $k'/k$.

 \section{\textbf{The  $2$-Iwasawa module}}
 Let us recall some results that will be useful in what follows.

 \begin{lemma}[\cite{chemsZkhnin2}]\label{lm deco to 2}
 	Let $m\geq1$ and    $p$  a  prime such that $p\equiv 3$ or $ 5  \pmod 8$. Then $p$ decomposes into the product of  $2$ prime ideals of $K_n=\mathbb Q{(\zeta_{2^{n+2}})}$ while it is inert in 
 	$K_n^+$.
 \end{lemma}

 %

 \begin{theorem}[\cite{kida}]\label{Kida}
 	Let $F$ and $K$ be   $CM$-fields and $K/F$ a finite $2$-extension. Assume that $\mu^-(F)=0$. Then $\mu^-(K)=0$ and
 	\[\lambda^-(K)-\delta(K)=[K_{\infty}:F_{\infty}]\left(\lambda^-(F)-\delta(F)\right)+\sum (e_{\beta}-1)-\sum(e_{\beta^+}-1),\]
 	where $\delta(k)$ takes the values $1$ or $0$ according to whether  $F_\infty$ contains the fourth roots of unity or not. The $e_{\beta}$ is the ramification index of a prime $\beta$ in $K_{\infty}/F_{\infty}$ coprime to $2$ and $e_{\beta^+}$ is the ramification index for a prime coprime to $2$ in $K_{\infty}^+/F_{\infty}^+$.
 \end{theorem}

 Let $rank_\ell (G)$ denote the $\ell$-rank of an abelian group $G$ and  $\mu(k)$,   $\lambda(k)$  denote the Iwasawa Invariants of $k$. We have:

  \begin{theorem}\label{thm rank zp}
 	Let $K_\infty/k$ be a $\ZZ_\ell$-extension of a  number field $k$ and assume that any prime which is ramified in 
 	$K_\infty/k$ is totally ramified. If $\mu(k)=0$ and  $\varprojlim A_n$ is an elementary $\Lambda$-module, then 
 	$rank_\ell(A_n)=\lambda(k)$ for all $n\geq \lambda(k)$. 		
 \end{theorem}
 
 \begin{proof}
 	We have $rank_\ell(A_n)=\lambda(k)$, for $n$ large enough and we have $rank_\ell(A_{n})\leq \lambda(k)$, for all $n$.
 	Assume that $rank_\ell(A_{m_0})\leq \lambda(k)-1$ for some $m_0\geq \lambda(k)$. This implies that necessary there is $r\leq  \lambda(k)$ 	such that 	
 	$rank_\ell(A_{r})=rank_\ell(A_{r+1})$ and therefore by \cite[Theorem 1]{fukuda},  $rank_\ell(A_{n})=rank_\ell(A_{k+1})\leq \lambda(k)-1$, for all $n\geq r$. Which is absurd.
 \end{proof}



 

 \begin{lemma}\label{lm parity of cn of real fields}
 	Let $p$  and $q$  be two primes satisfying conditions \eqref{CONDD1} or \eqref{COND2D}. Set $\pi_1= 2$, $\pi_2=2+\sqrt{2}$,..., $\pi_n=2+\sqrt{\pi_{n-1}}$. Then,     the $2$-class group of  $\mathbb{Q}(\sqrt{p}, \sqrt{q}, \sqrt{\pi_n})$ is trivial for all $n\geq 1$.
 \end{lemma}
 \begin{proof}
 	Assume that $p\equiv5\pmod 8$.
 	Put $k=\mathbb{Q}(\sqrt{p},\sqrt{q})$. If $\left(\dfrac{p}{q}\right)=-1$, then by   \cite[Lemma 3.2]{CZA-UH} we easily deduce that $q(k)=2$ (we similarly show that we have the same equality for other cases). Note that 
 	by class number formula (cf. \cite{wada})  and the fact that we have $h_2(p)=h_2(q)=1$ and 
 	$h_2(pq)=2$  (cf.   \cite[Corollaries 18.4,  19.7 and    19.8]{connor88}). we have
 	\begin{eqnarray*}
 		h_2(k)&=&\frac{1}{4}q(k)h_2(p)h_2(q)h_2(pq)=\frac{1}{4} \cdot 2 \cdot 1 \cdot 1 \cdot 2=1.  
 	\end{eqnarray*}
 	
 	By \cite[Theorem 4.4]{CZA-UH}  and \cite[Theorem 3.14]{ChemsTriqua}  the class number of $k_1=\mathbb{Q}(\sqrt{p}, \sqrt{q}, \sqrt{2})$ is odd. Therefore the result follows by \cite[Theorem 1]{fukuda}. We  prove similarly the result for the case  $p\equiv 3\pmod 8$.
 \end{proof}

 \begin{theorem}\label{thm main}
 	Let $p$  and $q$  be two primes satisfying conditions \eqref{CONDD1} or \eqref{COND2D}. Put $F=\QQ(\sqrt{p},\sqrt{q}, i)$. Let $A_n(F)$ denote the $2$-class group of the $n$-th layer of the cyclotomic $\ZZ_2$-extension of $F$.
 	Then 
 	the Iwasawa module $A_\infty(F)=\varprojlim A_n(F)$ of $F$ is a $\Lambda$-module without finite part.
 	If moreover $q\equiv 7\pmod{16}$, then 
 	\begin{enumerate}[\rm \indent $\bullet$]
 		\item We have $	A_\infty(F) \simeq \ZZ_2^3.$
 		
 		\item $rank_2(A_n(F))=3$, for all $n\geq 3 $.
 	\end{enumerate}
 	
 \end{theorem}	
 \begin{proof}	Assume that $p\equiv5\pmod 8$.
 	Let $F_n$ denote the $n$-th ($n\geq 1$) layer of the cyclotomic $\ZZ_2$-extension of $F$ and $A_n(F)$ denote the $2$-class group of $F_n$.
 	Let us define $A_n^+$ as the group of strongly ambiguous classes with respect to the extension $F_n/F_n^+$ and $A_n^-=A_n/A_n^+$ (cf. \cite{Katharina} for more details about this new definition). Since the class number of 
 	$F_n^+$ is odd (Lemma \ref{lm parity of cn of real fields}), then $A_n^+$ is generated by the ramified primes above $2$. Since $F_1/F_1^+$ is unramified, then  $F_n/F_n^+$ is also unramified. Therefore $A_n^+$ is trivial. So $A_n=A_n^-$. By  \cite[Theorem 2.5]{Katharina} there is no finite submodule in $  A_\infty^-$. 
 	Hence $A_\infty(F) =A_\infty^-$ is a $\Lambda$-module without finite part.\\
 	\noindent $\bullet$ Assume now that $q\equiv 7\pmod{16}$.
 	
 	Set $K_n=\mathbb Q{(\zeta_{2^{n+2}})}$  and $K=\mathbb{Q}(\sqrt{q}, i)$.   
 	Note that $p$ splits into $2$ prime ideals in  $\mathbb{Q}(\sqrt{q})$ or $\mathbb{Q}(\sqrt{2q})$.
 	Since by Lemma \ref{lm deco to 2}, $p$ splits into $2$ primes of $K_n$ and inert   in $K_n^+$, for $n$ large enough, then
 	$p$ splits into $4$ primes in $F_n$   while it splits into $2$ primes in $F_n^+$.  By \cite[Theorem 4]{chemskatharina} we have $\lambda^-(K)=1$.
 	Since $[F_{\infty}:K_{\infty}]=[F_{\infty}^+:K_{\infty}^+]=2$,	then by Kida's  formula (Theorem \ref{Kida}) applied on $F/K$ we have 
 	\[\lambda^-(F)-1=2(1-1)+4-2.\]
 	Thus by Lemma \ref{lm parity of cn of real fields} $\lambda(F)=\lambda^-(F)=3$. It follows that
 	$$A_\infty(F) \sim \bigoplus  \Lambda/(f_i(T)),$$
 	for some distinguished polynomials $f_i$ such that  $\sum \mathrm{deg}(f_i)=3$. Since by the division algorithm we have $ \Lambda/(f_i(T))\simeq \ZZ_2^{\mathrm{deg}(f_i)}$, then
 	$A_\infty(F)\simeq\ZZ_2^{3}$. Therefore, Theorem \ref{thm rank zp} completes the proof of the result for  $p\equiv5\pmod 8$. We   proceed similarly for the other cases.
  \end{proof}

\end{document}